\documentclass{article}
\usepackage[utf8]{inputenc}
\usepackage{amsmath,amssymb}
\usepackage{amsthm}
\usepackage{parskip}
\usepackage{hyperref}
\usepackage[margin=1.6in]{geometry}
\usepackage{authblk} 
\usepackage{mathtools} 
\mathtoolsset{showonlyrefs,showmanualtags}

\makeatletter
\def\thm@space@setup{%
  \thm@preskip=\parskip \thm@postskip=0pt
}
\makeatother

\hypersetup{
    colorlinks=true,
    linkcolor=blue,
    citecolor=darkgray
}

\usepackage{mathrsfs}
\usepackage{tikz}
\newtheorem{thm}{Theorem}[section]

\newtheorem{lemma}{Lemma}[section]

\newtheorem{remark}{Remark}
\theoremstyle{definition}

\usepackage[colorinlistoftodos]{todonotes}

\newcommand{\Res}{\text{Res}}

\newcommand{\e}{\varepsilon}
\newcommand{\cV}{\mathcal{V}}
\newcommand{\cD}{\mathcal{D}}

\newcommand{\C}{\mathbb{C}}

\newcommand{\ML}{\text{ML}}

\newcommand{\be}{\begin{equation}}
\newcommand{\ee}{\end{equation}}

\DeclareRobustCommand{\stirling}{\genfrac\{\}{0pt}{}}

\title{Asymptotic Enumeration of Lonesum Matrices}

\author[1]{Jessica Khera}
\author[2]{Erik Lundberg}
\author[3]{Stephen Melczer}
\affil[1]{Florida Atlantic University, jthune@fau.edu}
\affil[2]{Florida Atlantic University, elundber@fau.edu}
\affil[3]{University of Waterloo, smelczer@uwaterloo.ca }

\begin{document}

\maketitle

\begin{abstract}
\noindent We provide bivariate asymptotics for the poly-Bernoulli numbers, a combinatorial array that enumerates lonesum matrices, using the methods of Analytic Combinatorics in Several Variables (ACSV).
    For the diagonal asymptotic 
    (i.e., for the special case of square lonesum matrices)
    we present an alternative proof based on Parseval's identity.
    In addition, we provide an application in Algebraic Statistics on the asymptotic ML-degree of the bivariate multinomial missing data problem, and strengthen an existing result on asymptotic enumeration of permutations having a specified excedance set.
    \\ Keywords: poly-Bernoulli numbers, lonesum matrices, generating function, analytic combinatorics, asymptotics. 
\end{abstract}



\section{Introduction}
The poly-Bernoulli numbers $B_n^{(r)}$ were introduced by Kaneko~\cite{Kaneko} in a number-theoretic setting, and explored by Arakawa and Kaneko~\cite{ArakawaKaneko1999} for a connection to zeta functions and multiple zeta values. Subsequently, the case $B_{n,k} := B_n^{(-k)}$, when the superscript index $r=-k$ is negative, has attracted attention due to several combinatorial interpretations. The first such observation, from Brewbaker~\cite{Brewbaker}, was that $B_{n,k}$ enumerates the number of $n \times k$ matrices with entries $0$ or $1$ that can be uniquely reconstructed from their row sums and column sums, known as \emph{lonesum matrices}. Ryser~\cite{Ryser} gave a classification in terms of forbidden minors: a matrix is lonesum if and only if it does not contain $\begin{psmallmatrix}1 & 0 \\ 0 & 1 \end{psmallmatrix}$ or $\begin{psmallmatrix}0 & 1 \\ 1 & 0 \end{psmallmatrix}$ as a submatrix. 

Several combinatorial interpretations of $B_{n,k}$ are presented by B\'enyi and Hajnal~\cite{benyi}, including \emph{Callan permutations} (permutations of $\{1,\dots,n+k\}$ in which the elements of the sets $\{1,\dots,n\}$ and $\{n+1,\dots,n+k\}$ each appear in increasing order), \emph{Vesztergombi permutations}\footnote{The enumeration of such permutations original appeared in Vesztergombi~\cite{VesztergombiPerm}, who gave an analytic proof using generating functions.} (permutations $\pi$ of $\{1,\dots,n+k\}$ satisfying $-k \leq \pi(i)-i \leq n$), and $n\times k$ \emph{$\Gamma$-free matrices} (zero-one matrices without $\begin{psmallmatrix}1 & 1 \\ 1 & 0\end{psmallmatrix}$ and $\begin{psmallmatrix}1 & 1 \\ 1 & 1\end{psmallmatrix}$ as submatrices). 

The poly-Bernoulli numbers also appear in graph enumeration. For instance, Cameron et al.~\cite{CameronPreprint} note that $B_{n,k}$ counts the number of orientations of the edges of the complete bipartite graph $K_{n,k}$ with no directed cycles. Indeed, there is a natural bijection between acyclic orientations on $K_{n,k}$ and lonesum matrices: given an acyclic orientation of $K_{n,k}$ define the $n \times k$ matrix whose rows and columns correspond to the vertex sets of sizes $n$ and $k$, with an entry of 0 or 1 denoting the orientation of each edge. In a complete bipartite graph, an orientation being acyclic is equivalent to not containing any cycles of length 4, and forbidding 4-cycles in the graph is equivalent to forbidding the $2 \times 2$ submatrices characterizing lonesum matrices.

The poly-Bernoulli numbers $B_{n,n}$ also appear in the theory of matrix Schubert varieties~\cite{FRS}, where they count the number of strata in a certain stratification of the space of $n \times n$ matrices. 

An early paper of Arakawa and Kaneko~\cite{Kaneko2} gives the explicit formula 
$$B_{n,k} = \sum_{m \geq 0} (m!)^2 \stirling{n+1}{m+1} \stirling{k+1}{m+1},$$
where $\stirling{r}{s}$ denotes the Stirling number of the second kind. Here we provide asymptotics for $B_{n,k}$ as the indices $n,k \rightarrow \infty$. Asymptotic enumeration of lonesum matrices is a natural problem to consider from a purely combinatorial point of view, but further motivation comes from the appearance of poly-Bernoulli numbers in applications from biology~\cite{Cai} and algebraic statistics~\cite{Sullivant}. The growth rate of the poly-Bernoulli numbers is of particular relevance in those studies, as we briefly explain in the next two paragraphs.

In \cite{Cai}, Letsou and Cai propose the so-called \emph{ratchet model}, a noncommutative biological model for regulation of gene expression in cells. As the authors show, their model can be reduced to a group action on the set of lonesum matrices, and they use this to determine how the model's information content scales with its size. The super-exponential growth rate of the poly-Bernoulli numbers is key to their results, since it shows that the information content of the ratchet model is sufficient to account for the observed range of possible gene expressions,
whereas previously studied models based on combinatorial logic suffer from an information bottleneck and have information content that scales merely exponentially (see \cite[Fig. 2]{Cai}). Our main result, Theorem \ref{thm:multivar} below,
gives asymptotics for $B_{n,k}$ and hence provides a more precise description of how information content of the rachet model scales with its size. 

In \cite{Sullivant}, Ho\c sten and Sullivant study the maximum likelihood degree (ML-degree), a notion of algebraic complexity of maximum likelihood estimation in statistics.  They relate the ML-degree of the bivariate multinomial missing data problem to enumeration of lonesum matrices with no all-zero rows or columns; they observe that the ML-degree increases exponentially in the size of one of the multinomial variables while holding the size of the other variable fixed, indicating a high level of algebraic complexity. We explore this ML-degree, and give precise bivariate asymptotics when the sizes of both multinomial variables increase, in Section \ref{sec:ML} below. 

B\'enyi and Hajnal~\cite{benyihajnal} give further combinatorial interpretations of poly-Bernoulli `relatives' $D_{n,k}$ and $C_{n,k}$, which enumerate lonesum matrices with restrictions on the appearance of all-$0$ rows or columns\footnote{The number $C_{n,k}$ enumerates lonesum matrices of size $n \times k$ that have no column with all zeros, and $D_{n,k}$ enumerates lonesum matrices of size $n \times k$ that have no row or column with all zeros.}. The diagonal asymptotic\footnote{Here and throughout this paper $\log t$ denotes the natural logarithm of $t$.}
\begin{equation}
\label{eq:corrected}
D_{k,k} \sim  \frac{(k!)^2}{4\cdot(\log 2)^{2k+1}}\sqrt{\frac{1}{k\pi(1-\log 2)}} \quad \text{as } k \rightarrow \infty,
\end{equation}
follows from bijections in~\cite{benyihajnal} along with the asymptotic result in~\cite{Lovasz}, where the authors use a clever application of Parseval's identity followed by Cauchy's coefficient formula along with Laplace's method for asymptotic analysis of integrals.
In connection with the Algebraic Statistics problem mentioned above, we provide a more general bivariate asymptotic for $D_{n,k}$ in Section \ref{sec:ML} below.

\begin{remark}\label{rmk:mistake}
There is a small mistake in the final step (applying Laplace's method) of the proof of Theorem 1 in \cite{Lovasz}, with the stated result of that source missing a factor of 
$$\frac{1}{4 \sqrt{2 k} \log  2}.$$
The corrected form of the asymptotic is stated above in~\eqref{eq:corrected}.
\end{remark}

Asymptotics for the poly-Bernoulli relative $C_{n,k}$ follow from the bijections in \cite{benyihajnal} along with the asymptotics of certain permutation statistics given in \cite{ALN}, where the authors applied the general machinery of analytic combinatorics in several variables (ACSV)~\cite{PemantleWilson}, \cite{Baryshnikov} developed by Pemantle, Wilson, and Baryshnikov. We strengthen this result in Section \ref{sec:exc} below.

Here we use analytic methods to determine asymptotics for the standard poly-Bernoulli numbers $B_{n,k}$. We first use the more classical techniques from~\cite{Lovasz} to prove Theorem \ref{thm:diag} in Section \ref{sec:diag}; this method only applies to the diagonal case $n=k$ corresponding to square lonesum matrices.  We then adapt the method from \cite{ALN} in order to establish a more general bivariate asymptotic when $n,k \rightarrow \infty$ with $n/k$ varying within an arbitrary compact subset of the positive real numbers; see Theorem \ref{thm:multivar} in Section \ref{sec:multivar}. 
Other applications of the multivariate machinery appear in recent work on lattice path enumeration \cite{MelczerMishna2016}, \cite{MelczerWilson2019}.

Even though Theorem \ref{thm:diag} is a special case of Theorem \ref{thm:multivar}, the proof of Theorem \ref{thm:diag} given in Section \ref{sec:diag} provides an alternative perspective viewing a sum of squares as a Parseval identity, and this method may extend to other counting sequences whose terms can be expressed as a sum of squares.

\noindent {\bf Acknowledgments.}
We thank Seth Sullivant for helpful discussions and for pointing our attention to the reference \cite{FRS}. We also thank the anonymous referee for helpful comments.

\section{Asymptotics for \texorpdfstring{$k \times k$}{k x k} lonesum matrices}
\label{sec:diag}

\begin{thm}\label{thm:diag}
The number $B(k)=B_{k,k}$ of $k \times k$ lonesum matrices asymptotically satisfies
$$B(k) \sim (k!)^2 \sqrt{\frac{1}{k\pi(1-\log 2)}}\left( \frac{1}{\log 2} \right) ^{2k+1}, \quad \text{as \, } k \rightarrow \infty. $$
\end{thm}

\begin{proof}
As indicated in the introduction, our proof is inspired by the work of Lovasz and Vesztergombi \cite{Lovasz}. First, we interpret 
\begin{equation}\label{eq:n=k}
B(k) = \sum_{m=0}^k (m!)^2  \stirling{k+1}{m+1}^2
\end{equation}
as a Parseval formula\footnote{Or perhaps ``Pythagorean identity'' is more apt since~\eqref{eq:n=k} is a finite sum of squares.},
\begin{equation}
\label{eq:parseval}
B(k)=\frac{1}{2\pi} \int_{-\pi}^{\pi} \vert u_k(e^{i \varphi} ) \vert ^2 d\varphi
\end{equation}
where $$u_k(y)=\sum_{m=0}^k m! \stirling{k+1}{m+1}y^m .$$
Set 
$$\psi(x,y) = \sum_{k=0}^{\infty} u_k(y) \frac{x^k}{k!}.$$
Then $\psi$ can be expressed in closed form,
\begin{align*}
    \psi(x,y) & = \sum_{k=0}^{\infty} \frac{x^k}{k!} \sum_{m=0}^k m! 
    \stirling{k+1}{m+1}y^m \\
    & = \sum_{m=0}^{\infty} m!y^m \sum_{k=m}^{\infty} \stirling{k+1}{m+1} \frac{x^k}{k!} \\
    & = \sum_{m=0}^{\infty} y^m e^x(e^x-1)^m \\
    &= \frac{e^x}{1-y(e^x-1)},
\end{align*}
where we have used the identity
$$\sum_{k=m}^{\infty} \stirling{k+1}{m+1} \frac{x^k}{k!} = \frac{e^x(e^x - 1)^m}{m!},$$
which follows from shifting the index and then differentiating the more basic identity \cite[Sec. 1.4]{Stanley}
\be\label{eq:morebasic}
\sum_{k=m}^{\infty} \stirling{k}{m} \frac{x^{k}}{k!} = \frac{(e^x - 1)^m}{m!}.
\ee

Using the Cauchy Integral Formula, we have 
\begin{equation}\label{eq:uk}
u_k(y) = \frac{k!}{2 \pi i } \oint_{C_\e} \frac{\psi(x,y)}{x^{k+1}} dx , \end{equation}
where the contour $C_\e$ is a small circle $|x|=\e$ traced counterclockwise. For any $y\in\mathbb{C}$ the pole of $\psi$ closest to the origin is $x=x_0=\log(1+1/y)$. 

Let $\Gamma$ denote the contour consisting of the three segments 
$$\Gamma_- := [-\infty-\pi i,2-\pi i], \, \Gamma_0 := [2-\pi i, 2+ \pi i], \text{ and } \Gamma_+ := [2+\pi i , -\infty + \pi i].$$
For $|y|=1$, the contour $\Gamma$ surrounds $x_0$ and no other singularities of $\psi(x,y)$,
and we have
\begin{equation}\label{eq:CauchyThm}
\oint_{\Gamma} \frac{\psi(x,y)}{x^{k+1}} dx
-  \oint_{C_\e} \frac{\psi(x,y)}{x^{k+1}} dx = 2 \pi i \, \Res_{x=x_0}  \frac{\psi(x,y)}{x^{k+1}}. \end{equation}

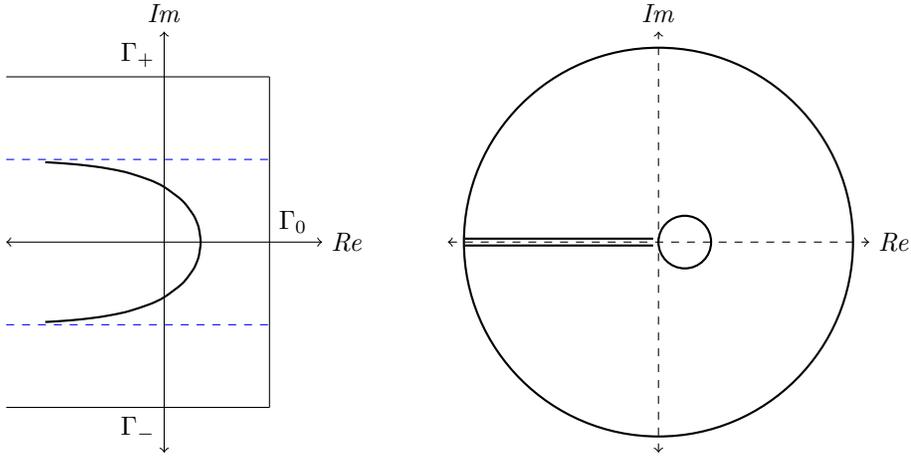
\begin{figure}

\begin{tikzpicture}[scale=0.7]
    \draw[<->] (-3,0) -- (3,0) node[right] {\textit{Re}};
    \draw[<->] (0,-4) -- (0,4) node[above] {\textit{Im}};
    \draw (-3, -3.1415) -- node[below]{$\Gamma_{-}$} (2, -3.1415);
    \draw (-3, 3.1415) -- node[above]{$\Gamma_{+}$} (2, 3.1415);
    \draw (2, -3.1415) --  node[above right]{$\Gamma_0$} (2, 3.1415) ;
    \draw[dashed,domain=-3:2,smooth,variable=\x,blue] plot (\x,1.57079);
    \draw[dashed,domain=-3:2,smooth,variable=\x,blue] plot (\x,-1.57079);
 
    \draw[thick] plot [smooth] coordinates {
    (-2.256924, -1.518436)
    (-2.034037, -1.505346)
    (-1.746911 , -1.483529)
    (-1.496426, -1.458596)
    (-1.256566, -1.427996)
    (-1.1619709, -1.413716)
    (-0.941145, -1.374446)
    (-0.809586, -1.346397)
    (-0.658478,-1.308997) 
    (-0.267399, -1.178097)
    (0,-1.0471975) 
    (0.34657,-0.785398) 
    (0.549306, -0.523598) 
    (0.6139735, -0.392699)
    (0.658478, -0.261799)
    (0.693147,0) 
    (0.658478, 0.261799)
    (0.6139735, 0.392699) 
    (0.549306, 0.523598)  
    (0.34657, 0.785398)
    (0, 1.0471975) 
    (-0.267399, 1.178097)  
    (-0.658478, 1.308997)
    (-0.809586, 1.346397)
    (-0.941145, 1.374446) 
    (-1.1619709, 1.413716) 
    (-1.256566, 1.427996)
    (-1.496426, 1.458596)
    (-1.746911 , 1.483529)
    (-2.034037, 1.505346)
    (-2.256924, 1.518436)
    };
\end{tikzpicture}
\hspace{0.3in}
\begin{tikzpicture}[scale=0.35]
    \draw[<->,dashed] (-8,0) -- (8,0) node[right] {\textit{Re}};
    \draw[<->,dashed] (0,-8) -- (0,8) node[above] {\textit{Im}};
    \draw[thick][-] (-7.35,0.13) -- (-0.2,0.13);
    \draw[thick][-] (-7.35,-0.13) -- (-0.2,-0.13);
    \draw[thick] (1,0) circle (1);
    \draw[thick] (0,0) circle (7.389056) ;
\end{tikzpicture}
\caption{ Left: Part of the unbounded contour $\Gamma = \Gamma_- \cup \Gamma_0 \cup \Gamma_+$ is shown along with the unbounded set $\{ \log(1+1/y) : |y|=1 \}$ of locations of the nearest (to the origin) singularity $x_0$.
Right: The image of each of these under the exponential function.} 
\label{fig1}
\end{figure} 

This follows from the residue theorem, if in place of $\Gamma$ we use a finite rectangular contour $\Gamma_R$
obtained from truncating $\Gamma$ at $\Re \, x = -R$, with $R > - \Re x_0$,
and inserting a left edge at $\Re \, x = -R$ to form a rectangle.
In order to arrive at the statement using the contour $\Gamma$ we deform the contour $\Gamma_R$, letting $R \rightarrow \infty$.
This requires an estimate; we notice that the integrand is $O(|x|^{-k-1})$ for $\Re x \rightarrow \infty$ with $|\Im x| \leq \pi$, and for $k \geq 1$ this is sufficient to justify deforming the contour $\Gamma_R$ to arrive at the infinite contour $\Gamma$.

We can thus rewrite~\eqref{eq:uk} as 
\begin{align*}
\frac{u_k(y)}{k!} &= -\Res_{x=x_0} \frac{\psi(x,y)}{x^{k+1}} + \frac{1}{2\pi i}\oint_{\Gamma} \frac{\psi(x,y)}{x^{k+1}} dx.
\end{align*}
For $y = e^{i\varphi}$ we have $x_0 = \log (1+e^{-i\varphi})$, and from the above we obtain
\begin{equation}\label{eq:res}
\frac{u_k(y)}{k!} = \frac{1}{e^{i \varphi}(\log  (1+e^{-i \varphi}))^{k+1}}  + O(2^{-k}),
\end{equation}
where we have computed the residue by evaluating (see \cite[p. 151]{Ahlfors})
$$ \lim_{x \rightarrow x_0} (x-x_0)\frac{\psi(x,e^{i \varphi})}{x^{k+1}}, \quad x_0 = \log(1+e^{-i\varphi}). $$
The above estimate 
\begin{equation}\label{eq:est}
 \oint_{\Gamma} \frac{\psi(x,y)}{x^{k+1}} dx = O(2^{-k})
\end{equation}
can be seen as follows.
For $|y|=1$, we have
\begin{align*}
    |\psi(x,y)| &= \frac{|e^x|}{|y|\left| 1+\frac{1}{y}-e^x\right|}  \\
    &= \frac{|e^x|}{\left| 1+\frac{1}{y}-e^x\right|}.
\end{align*}
Along $\Gamma_{\pm}$
we then have (see Figure \ref{fig1})
$$|\psi(x,y)| \leq \frac{|e^x|}{|e^x|} =1 ,$$ 
and along $\Gamma_0$ we have
$$ |\psi(x,y)| \leq
\frac{|e^x|}{|e^x|-\left|1+1/y\right|} \leq
\frac{e^2}{e^2-2},$$
so that the estimate \eqref{eq:est}
is reduced to showing
\begin{equation}\label{eq:estred}
 \oint_{\Gamma} 
 |x|^{-k-1} |dx| = O(2^{-k}).
\end{equation}
We have
\begin{align*}
 \oint_{\Gamma } |x|^{-k-1} |dx| &= 
 \oint_{\Gamma \cap \{ \Re \, x \geq -2 \}} |x|^{-k-1} |dx| + \oint_{\Gamma \cap \{ \Re \, x \leq -2 \}} |x|^{-k-1} |dx| \\
 &\leq (2\pi+2 \cdot 4) 2^{-k-1} + 2 \int_{2}^\infty t^{-k-1} dt \\
 &=(2\pi+8) 2^{-k-1} +  \frac{2}{k}2^{-k} \\
 &= O(2^{-k}).
\end{align*}

 Substitution of \eqref{eq:res} into \eqref{eq:parseval} then gives
$$B(k) \sim \frac{(k!)^2}{2 \pi} \int_{-\pi}^{\pi} \frac{d\varphi}{ \vert \log  (1+e^{-i \varphi}) \vert ^{2k+2} } , \quad \text{as } k \rightarrow \infty .$$


In preparation for applying the Laplace method for asymptotic analysis of integrals, we rewrite the integrand as
$$\frac{1}{ \vert \log  (1+e^{-i \varphi}) \vert ^{2k+2} } = \exp \left\{ -(k+1) 2 \log  \left|  \log  (1+e^{-i \varphi}) \right|  \right\}.$$
Recall Laplace's method for asymptotics of such integrals~\cite[Sec. 4.2]{deBruijn}: if $f$ is a continuously differentiable function on $[a,b]$ and $\varphi_0 \in (a,b)$ is the unique point in $[a,b]$ where $f$ attains a global maximum, then 
$$ \int_a^b e^{tf(\varphi)} d \varphi 
\sim \int_{-\infty}^{\infty} e^{tf(\varphi_0) + tf''(\varphi_0)\varphi/2} d \varphi 
= \sqrt{\frac{2 \pi}{t \vert f''(\varphi_0) \vert}} e^{t f(\varphi_0)} \quad \text{as } t \rightarrow \infty .$$
Applying this to our situation, we notice that
$f(\varphi) = -2 \log  \left|  \log  (1+e^{-i \varphi}) \right| $ satisfies these conditions with $\varphi_0=0$, and setting $t=k+1$ we obtain
$$\frac{1}{2\pi}\int_{-\pi}^{\pi} \frac{d\varphi}{ \vert \log  (1+e^{-i \varphi}) \vert ^{2k+2} } \sim   \sqrt{\frac{1}{k\pi(1-\log 2)}} \left( \frac{1}{\log 2} \right)^{2k+1} \quad \text{as } k \rightarrow \infty.$$
Multiplying by $(k!)^2$ we arrive at the desired result.
\end{proof}

\begin{remark}\label{rmk:expansion}
The above readily yields additional terms in the asymptotic expansion for $B_{k,k}$ by using the extension of Laplace's method provided in \cite[Sec. 4.4]{deBruijn}. To illustrate, we state the asymptotic with two orders of precision:
$$B(k) =   \frac{(k!)^2}{\sqrt{(k+1)\pi(1-\log 2)}} \left( \frac{1}{\log 2} \right) ^{2k+1}   \left( 1 + \frac{C}{k} + O\left( k^{-2} \right) \right), \quad \text{as \, } k \rightarrow \infty, $$
where $C= \displaystyle \frac{2\log^3 2 + 3 \log^2 2 - 12\log 2 +6}{16 (1-\log 2)^2 }$.
\end{remark}

\section{Bivariate asymptotics for the number of \texorpdfstring{$n \times k$}{n x k} lonesum matrices}
\label{sec:multivar}

In this section we determine bivariate asymptotics of $B_{n,k}$. 
To state our main result we define the function
\begin{equation}\label{eq:f}
f(t) = \frac{t}{(1-e^t)\log(1-e^{-t})},
\end{equation}
which has strictly positive derivative, goes to 0 as $t\rightarrow0$, and goes to infinity as $t\rightarrow\infty$ (see \cite[Appendix]{ALN}). 
Thus, $f$ is a bijection from the positive real line $(0,\infty)$ to itself.

\begin{thm}\label{thm:multivar}
If $n,k\rightarrow\infty$ such that $n/k$ approaches a positive constant, and $a=a(n,k)= f^{-1}(n/k)$ and $b=b(n,k) = f^{-1}(k/n)$, then 
\begin{equation}\label{eq:bivarasymp}
B_{n,k} = \frac{a^{-n}b^{-k}}{\sqrt{k}} \, \frac{n! \, k!}{\sqrt{2\pi ae^{-a}\left[be^{-b} + a e^{-a} - ab \right]}}\left(1+O\left(k^{-1}\right)\right).\end{equation}
The implied constant in the big-O error term can be uniformly bounded as $n/k$ varies in any compact set of $\mathbb{R}_{>0}$.
\end{thm}

\begin{remark}
When $n=k$ we obtain the explicit formula 
$$ f^{-1}(k/n) = f^{-1}(n/k) = f^{-1}(1) = \log 2, $$
and we recover Theorem \ref{thm:diag} from Theorem \ref{thm:multivar}.
\end{remark}

Our argument begins with the bivariate exponential generating function
$$ F(x,y) := \sum_{n,k \geq 0} B_{n,k} \frac{x^n}{n!} \frac{y^k}{k!} = \frac{1}{e^{-x}+ e^{-y}-1}, $$
whose derivation can be found, for instance, in B{\'e}nyi and Hajnal~\cite{benyihajnal}. As in the univariate case the set of $F$'s singularities, here the set
\[ \cV = \{(x,y)\in\C^2 : H(x,y)=0 \}, \]
is crucial to determining asymptotics. A singularity $(p,q) \in \cV$ is called \emph{minimal} if there does not exist $(x,y) \in \C^2$ such that $H(x,y)=0$ with $|x| \leq |p|$ and $|y| \leq |q|$, with one of the inequalities being strict. Equivalently, the minimal singularities are the elements of $\cV$ on the closure of the power series domain of convergence $\cD$ of $F(x,y)$. A singularity $(p,q) \in \cV$ is \emph{strictly minimal} if it is minimal and there are no other singularities with the same coordinate-wise modulus.

\begin{lemma}
\label{lem:min}
A singularity $(a,b) \in \cV$ is minimal if and only if $a$ (and thus also $b$) is positive and real. In particular, every minimal singularity is strictly minimal.
\end{lemma}

\begin{proof}
This follows from application of results in \cite[Sec. 3]{PemWilTwenty}, but the arguments are short and elementary so we provide a self-contained proof for the convenience of the reader.

Since the coefficients $B_{n,k}$ are non-negative, for $(x,y)$ in the power series domain of convergence $\cD$ we have
\[ \left|F(x,y)\right| = \left|\sum_{n,k \geq 0} B_{n,k} \frac{x^n}{n!} \frac{y^k}{k!}\right| \leq \sum_{n,k \geq 0} B_{n,k} \frac{|x|^n}{n!} \frac{|y|^k}{k!} = F(|x|,|y|).\]
Because $F$ is the ratio of analytic functions, any singularity of $F$ is a polar singularity. In particular, if $(p,q) \in \cV$ is minimal then there exist a sequence of points $(p_n,q_n) \in \cD$ such that $|F(p_n,q_n)|\rightarrow\infty$. But then each $(|p_n|,|q_n|) \in \cD$ and $F(|p_n|,|q_n|)\rightarrow\infty$, so $(p,q) \in \cV$ is a minimal singularity only if $(|p|,|q|) \in \cV$ is also a minimal singularity.

Suppose now that $(a e^{i\theta_1}, b e^{i \theta_2}) \in \cV$ with $a,b>0$ is minimal, so that the last paragraph implies $(a,b) \in \cV$. 
Since $F(x,y) = \frac{e^{x+y}}{1-P(x,y)}$ where
\[   P(x,y) = 1-e^x - e^y + e^{x+y} = \sum_{k=1}^\infty \frac{1}{k!}\sum_{j=1}^{k-1} \binom{k}{j}x^j y^{k-j},\]
we have $P(a,b)=1=P(ae^{i\theta_1},b e^{i \theta_2})$ and
\[ \sum_{k=1}^\infty \frac{1}{k!}\sum_{j=1}^{k-1} \binom{k}{j}a^j b^{k-j} = \left|\sum_{k=1}^\infty \frac{1}{k!}\sum_{j=1}^{k-1} \binom{k}{j}a^jb^{k-j}e^{ji\theta_1} e^{(k-j)i\theta_2} \right|. \]
Equality holds only if equality holds term by term, so $e^{i n\theta_1} e^{i k \theta_2} = 1$ for each $n,k>0$ and thus $\theta_1 = \theta_2=0$. In particular, the only minimal points lie in the positive quadrant. 
\end{proof}

Theorem \ref{thm:multivar} is then an immediate consequence of standard results in the theory of analytic combinatorics in several variables.

\begin{thm}[{Pemantle and Wilson~\cite[Thm. 9.5.7]{PemantleWilson}}]  
\label{thm:PW}
Let $F(x,y)$ be the ratio of entire functions $G,H$ and let $\cV = \left\{(x, y) \in \C^2: H(x, y) = 0\right\}$. Suppose $F(x, y) = \sum_{r, s \geq 0} f_{r,s} x^ry^s$ is the power series expansion of $F$ at the origin and
\begin{enumerate}  
    \item[$(i)$] for each $r,s>0$ there exists a unique minimal point $(x_{r,s},y_{r,s}) \in \cV$ solving the system
    \begin{equation} H(x,y) = sxH_x(x,y) - ryH_y(x,y) = 0, \label{eq:crit} \end{equation}
    and this point is strictly minimal;
    \item[$(ii)$]  The point $(x_{r,s},y_{r,s})$ varies smoothly with $r,s$;
    \item[$(iii)$]  The point $(x_{r,s},y_{r,s})$ is not a root of the polynomial $G(x,y)$ nor a root of the polynomial
        $$Q(x,y):= -y^2H_y^2xH_x-yH_yx^2H_x^2-x^2y^2(H_y^2H_{xx}+H_x^2H_{yy}-2H_xH_yH_{xy}).  $$
\end{enumerate}  
Then as $r, s \to \infty$
\begin{equation}\label{eq:PW}
 A_{r,s} = \left( G(x,y) + O(s^{-1}) \right) \frac{1}{\sqrt{2\pi}}x^{-r}y^{-s}\sqrt{\frac{-yH_y(x,y)}{sQ(x,y)}},
\end{equation}
where the constant in the error term $O(s^{-1})$ can be uniformly bounded as $r/s$ varies in any compact set.
\end{thm}

Although Theorem~\ref{thm:multivar} is a direct application of Theorem~\ref{thm:PW}, we sketch the argument for our situation as this seems a nice opportunity to illustrate the basic methods involved in analytic combinatorics in several variables
(see also \cite{PemWilTwenty}, \cite{MelczerMishna2016} for illustrative presentations of the methods in ACSV). The starting point is a bivariate Cauchy integral representation
\begin{equation} B_{n,k} = \frac{n!k!}{(2\pi i)^2} \int_{|x|=u} \left(\int_{|y|=v} \frac{1}{e^{-x}+e^{-y}-1} \frac{dy}{y^{k+1}}\right)\frac{dx}{x^{n+1}}, \label{eq:mCIF} \end{equation}
where $T(u,v) = \{(x,y)\in\C^2 : |x|=u,|y|=v\}$ for any $(u,v) \in \cD$. If $(a,b)$ is a minimal point then $(u,v)$ in Equation~\eqref{eq:mCIF} can be replaced by $(a,b-\epsilon)$ for any sufficiently small $\epsilon>0$. When $(a,b)$ is strictly minimal then the domain of integration $|x|=a$ can be replaced by any neighbourhood $\mathcal{N}$ of $a$ in the circle $|x|=a$ while introducing an exponentially negligible error. Furthermore, if $(a,b)$ satisfies Equation~\eqref{eq:crit} then replacing $|y|=b$ by $|y|=b+\epsilon$ results in an integral which is also exponentially smaller than $B_{n,k}$. Thus, up to an exponentially negligible error the sequence of interest is a difference of integrals
\[ \frac{n!k!}{(2\pi i)^2}\int_{x \in \mathcal{N}} \left(\int_{|y|=b-\epsilon} \frac{1}{e^{-x}+e^{-y}-1} \frac{dy}{y^{k+1}} - \int_{|y|=b+\epsilon} \frac{1}{e^{-x}+e^{-y}-1} \frac{dy}{y^{k+1}}\right)\frac{dx}{x^{n+1}}.\]
When $(a,b)$ is strictly minimal, the inner difference of integrals equals the residue of the integrand at the singularity $y=-\log(1-e^{-x})$, meaning the sequence of interest is asymptotically approximated by the integral
\[ I_{n,k} = \frac{n!k!}{2\pi i}\int_{x \in \mathcal{N}} \frac{1}{1-e^{-x}} \frac{dx}{x^{n+1}\left(-\log(1-e^{-x})\right)^{k+1}}. \]
Parameterizing $\mathcal{N} = \{ae^{i\theta} : -\tau \leq \theta \leq \tau \}$ for some $\tau>0$ we obtain
\begin{equation} I_{n,k} = \frac{n!k!}{2\pi a^n}\int_{-\tau}^{\tau} \frac{1}{1-e^{-ae^{i\theta}} } e^{-ni\theta - (k+1)\log\left(-\log\left(1-e^{-ae^{i\theta}}\right)\right)} d\theta. \label{eq:finalSaddle} \end{equation}
When $(a,b)$ satisfies Equation~\eqref{eq:crit} then~\eqref{eq:finalSaddle} is a saddle-point integral. Replacing these functions by their leading power series terms (up to second order) at the origin gives the asymptotic approximation (see \cite[Ch. 5]{deBruijn} for an exposition of the saddle-point method)
\begin{align*} 
I_{n,k} &\sim \frac{n! k! e^b}{2\pi a^n b^{k+1}} \int_{-\tau}^{\tau} e^{-k\theta^2 \frac{ae^{-a}\left( -be^{-a}+ae^{-a}-ab+b \right)}{2e^{-2b} b^2}} d\theta \\
&\sim \frac{n! k! e^b}{2\pi a^n b^{k+1}} \int_{-\infty}^{\infty} e^{-k\theta^2 \frac{ae^{-a}\left( -be^{-a}+ae^{-a}-ab+b \right)}{2e^{-2b} b^2}} d\theta \\
&=\frac{a^{-n}b^{-k}}{\sqrt{k}} \times \frac{n!k!}{\sqrt{2 \pi ae^{-a}(be^{-b}+ae^{-a}-ab)}}
\end{align*}
stated in Theorem~\ref{thm:multivar}. As in Remark \ref{rmk:expansion} above, asymptotics of $B_{n,k}$ can be determined to larger order by using known formulas to compute additional terms in the asymptotic expansion of~\eqref{eq:finalSaddle}.

\section{Asymptotic enumeration of permutations with a single run of excedances}\label{sec:exc}

Lemma \ref{lem:min} also implies a strengthened version of \cite[Thm. 1.3]{ALN}, which gives asymptotics for a multivariate generating function of a similar form.

Given a permutation $\pi$ on
$\{1,2,...,n\}$, we say that $j$ is an \emph{excedance} of $\pi$ if $\pi(j) > j$,
and the \emph{excedance word} 
$w(\pi) = w_1 w_2 \cdots w_{n-1} \in \{A,B \}^{n-1}$ of $\pi$ is defined by
$w_j = B$ if $j$ is an excedance and $w_j = A$ otherwise.
For an excedance word $w \in \{A,B\}^{n-1}$,
the bracket $[w]$ denotes the number of permutations with excedance word $w$ (see \cite{EhrSt} for more details).

\begin{thm}
If $r,s\rightarrow\infty$ such that $r/s$ approaches a positive constant, and $x=x(r,s)= f^{-1}(r/s)$ and $y=y(r,s) = f^{-1}(s/r)$, then 
\begin{equation}
[B^{r-1}A^s] = \frac{x^{-r}y^{-s}}{\sqrt{2\pi s}} \, \frac{r! \, s! \, e^{-y}}{\sqrt{ xe^{-x}\left[ye^{-y} + x e^{-x} - xy \right]}}\left(1+O\left(s^{-1}\right)\right).\end{equation}
The implied constant in the big-O error term can be uniformly bounded as $r/s$ varies in any compact set of $\mathbb{R}_{>0}$.
\end{thm}

This result is proved in \cite{ALN}
under the additional assumption that
$(r,s)$ stays within a certain sector-shaped neighborhood of the diagonal $r=s$.

\begin{proof}
The result follows from the proof given in \cite{ALN}
while using the above Lemma \ref{lem:min} in place of \cite[Lemma 3.1]{ALN}.
\end{proof}

\section{An application in Algebraic Statistics}\label{sec:ML}

In this section we provide asymptotics for the maximum likelihood degree (ML-degree) of the bivariate multinomial missing data problem. This statistical problem concerns estimating the parameters in a multivariate probability model given some observations with missing data (several replicates each with some missing covariates).
A standard approach for estimating the parameters uses maximum likelihood estimation, which involves maximizing an associated log-likelihood function. The ML-degree is a measure of the algebraic complexity of this method, defined as the number of complex solutions of the \emph{critical equations} obtained by setting the gradient of the log-likelihood function to zero. In the case of the bivariate multinomial missing data problem, where $X_1$ and $X_2$ are discrete multinomial random variables with $X_1 \in \{ 1, 2, \dots, n\}$ and $X_2 \in \{1, 2, \dots, k\}$, the ML-degree turns out to equal the number of bounded regions in the complement of a certain arrangement of hyperplanes. Counting the number of such regions reduces to enumerating lonesum matrices of size $n \times k$ having no all-zero rows or columns. 

In \cite[Thm. 7.3]{Sullivant} the ML-degree of the bivariate multinomial $n \times k$ missing
data problem is expressed using the inclusion-exclusion principle leading to
\be\label{eq:ML}
\ML(n,k) =  \sum_{m=0}^n\sum_{\ell=0}^k (-1)^{m+\ell} \binom{n}{m}\binom{k}{\ell} B_{n-m, k-\ell}.
\ee 
Since the ML-degree of the bivariate multinomial missing data problem can be reduced to counting the number $D_{n,k}$ of lonesum matrices of size $n \times k$ with no all-zero rows or columns, we can alternatively use \cite[Thm. 2]{benyihajnal} to express the ML-degree by
\be\label{eq:MLsimple}
\ML(n,k) = \sum_{m \geq 0} (m!)^2 \stirling{n}{m} \stirling{k}{m}.
\ee 
As noted in~\cite[Thm. 3]{benyihajnal}, this representation directly yields the exponential generating function for $\ML(n,k)$.

\begin{thm}\label{thm:MLgf}
The exponential generating function
\be\label{eq:MLgf}
M(x,y) = \sum_{n,k \geq 0} \textsl{ML}(n,k) \frac{x^n y^k}{n! k!}
\ee
for $\textsl{ML}(n,k)$ satisfies
\be\label{eq:MLgfanalytic}
M(x,y) = \frac{e^{-x-y}}{e^{-x}+e^{-y}-1}.
\ee
\end{thm}

Since $\ML(n,k)=D_{n,k},$ asymptotics for $\ML(n,k)$ on the diagonal $n=k$ are provided by~\eqref{eq:corrected}.
Another application of Theorem~\ref{thm:PW} gives the general bivariate asymptotic.

\begin{thm}\label{thm:MLasymp}
If $n,k \rightarrow\infty$ such that $n/k$ approaches a positive constant, and $a=a(n,k)= f^{-1}(n/k)$ and $b=b(n,k) = f^{-1}(k/n)$, then 
\begin{equation}
ML(n,k) = \frac{a^{-n}b^{-k}}{\sqrt{2\pi k}} \, \frac{n! \, k! \, e^{-a-b}}{\sqrt{a e^{-a}\left[be^{-b} + a e^{-a} - ab \right]}}\left(1+O\left(k^{-1}\right)\right).\end{equation}
The implied constant in the big-O error term can be uniformly bounded as $n/k$ varies in any compact subset of $\mathbb{R}_{>0}$.
\end{thm}

\begin{proof}
This is an application of Lemma \ref{lem:min} and Theorem \ref{thm:PW} to the exponential generating function provided by Theorem \ref{thm:MLgf}. Because the denominator of the rational function under consideration is the same as that analyzed in Theorem \ref{thm:multivar}, the proof of Theorem~\ref{thm:MLasymp} is identical to the proof of Theorem~\ref{thm:MLasymp} aside from the modification that the numerator is now $G(x,y) = e^{-x-y}$.
\end{proof}

\section{Local Central Limit Theorems}

In addition to the above asymptotic results, our approach allows for the calculation of local central limit theorems. As an example, we prove that fixing the index $n$ causes the distribution of our sequence terms as a function of $k$ to approach the density of a normal distribution, when appropriately scaled.

To this end, and to highlight the variability of $k$ for fixed $n$, let $b_n(k)$ and $d_n(k)$ denote the coefficients of $x^ny^k$ in the power series expansions of $1/(e^{-x}+e^{-y}-1)$ and \mbox{$e^{-x-y}/(e^{-x}+e^{-y}-1)$} at the origin, respectively. In particular, $b_n(k) = B_{n,k}/(n!k!)$ and~$d_n(k) = D_{n,k}/(n!k!)$.

\begin{thm}[Local Central Limit Theorems]
\label{thm:LCLT}
The limits
\begin{align*}
\sqrt{n} \, \sup_{k \in \mathbb{N}}{\Big|} \rho^n b_n(k) - \nu_n(k) {\Big|} &\rightarrow 0 
\\[+2mm]
\sqrt{n} \, \sup_{k \in \mathbb{N}}{\Big|} \rho^n d_n(k) - e^{-1}(1-e^{-1})\nu_n(k) {\Big|} &\rightarrow 0
\end{align*}
hold as $n\rightarrow\infty$, where 
\[ \rho = 1-\log(e-1) = 0.458\dots\] 
and $\nu_n(k)$ is the Gaussian density
\[ \nu_n(k) = \frac{C}{\sqrt{2\pi n \sigma}}e^{\frac{-(k-n\omega)^2}{2\sigma n}} \]
with parameters 
\begin{align*}
C &= \frac{e}{\left(1-\log(e-1)\right)(e-1)} = 3.449\dots\\[+2mm]
\omega &= \frac{1}{\left(1-\log(e-1)\right)(e-1)} = 1.268\dots \\[+2mm]
\sigma &= \frac{\log(e-1)}{\left(1-\log(e-1)\right)^2(e-1)^2} = 0.871\dots.
\end{align*}
\end{thm}

\begin{figure}[t]
\centering
\includegraphics[width=0.8\linewidth]{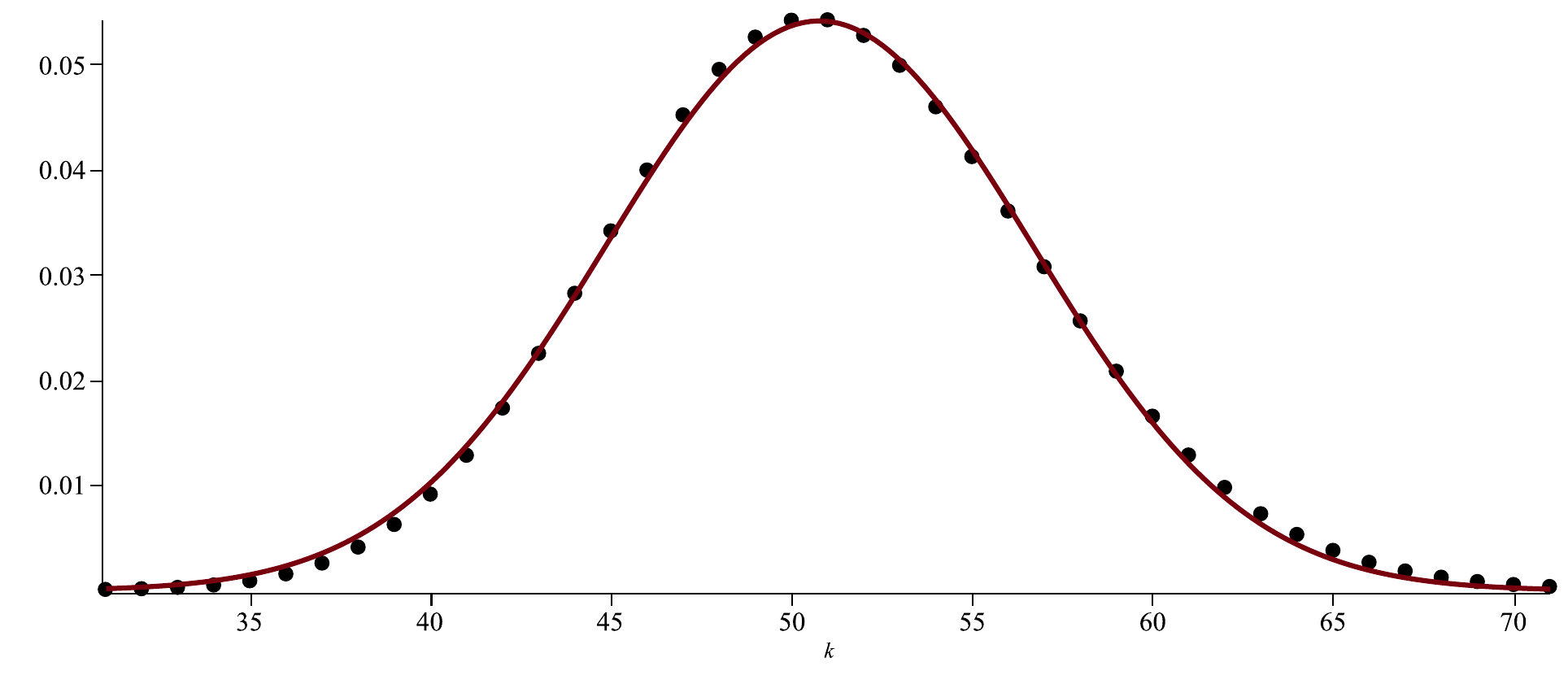}
\caption{Plot of the power series coefficients $x^{40}y^k$ of \mbox{$\rho^{40}e^{-x-y}/(e^{-x}+e^{-y}-1)$}, where $\rho=1-\log(e-1)$, and the limit density described by Theorem~\ref{thm:LCLT}.}
\label{fig:LCLT}
\end{figure}

\begin{proof}
Our analysis of the singular sets of $1/(e^{-x}+e^{-y}-1)$ and \mbox{$e^{-x-y}/(e^{-x}+e^{-y}-1)$} above allow us to apply Theorem 9.6.6 of Pemantle and Wilson~\cite{PemantleWilson}, which yields the stated limit density.
\end{proof}

Theorem~\ref{thm:LCLT} is illustrated by Figure~\ref{fig:LCLT}. With a small amount of additional work we can also derive a limit theorem for the ML-degree $\text{ML}(n-k,k)=D_{n-k,k}$ for $k$ near $n/2$.

\begin{thm}
\label{thm:MLlimit}
For all fixed $K>0$,
$$ \sup_{|k - n/2|  \leq K \sqrt{n}} \left| \frac{(2\log 2)^{n}}{n!} \text{ML}(n-k,k) - \frac{2^{-\frac{2(k-n/2)^2}{n(1-\log 2)}}}{(4\log 2)\sqrt{1-\log 2}}\right| \rightarrow 0 $$
as $n\rightarrow\infty$.
\end{thm}

\begin{figure}[t]
\centering
\includegraphics[width=0.9\linewidth]{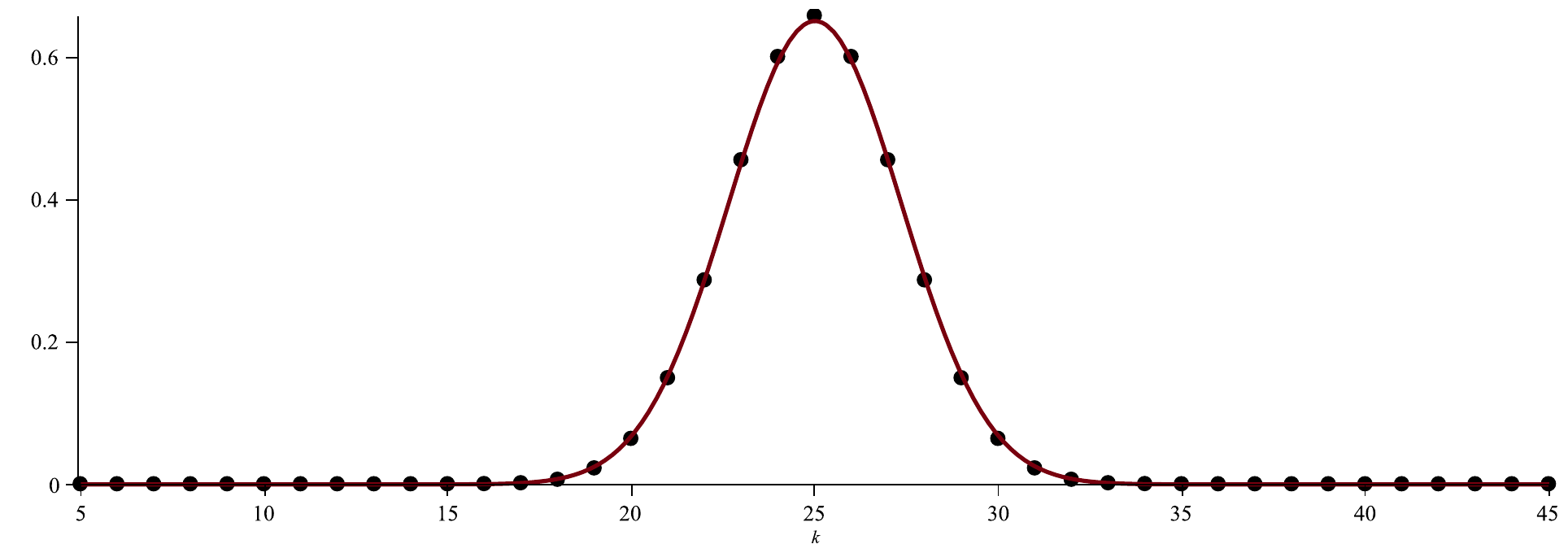}
\caption{Plot of $\frac{(2\log 2)^{n}}{n!} \text{ML}(n-k,k)$ compared to the limit shape described in Theorem~\ref{thm:MLlimit}, when $n=50$.}
\label{fig:ML-LCLT}
\end{figure}

Theorem~\ref{thm:MLlimit} is illustrated in Figure~\ref{fig:ML-LCLT}.

\begin{proof}
Making the change of variables $y=xu$ in the generating function $M(x,y)$ described by Theorem~\ref{thm:MLgf} yields
\begin{align*}
    M(x,xu) &= \sum_{N=0}^\infty \sum_{k=0}^N D_{N-k,k} \frac{x^N u^k}{(N-k)!k!} \\
    &= \sum_{N=0}^\infty \sum_{k=0}^N D_{N-k,k} \binom{N}{k} u^k \frac{x^N }{N!} .
\end{align*}
Since
$$M(x,xu) = \frac{e^{-x(u+1)}}{e^{-x}+e^{-xu}-1},$$
if we define $v_n(k) = \frac{(\log2)^n}{n!}\binom{n}{k} D_{n-k,k}$ then an argument analogous to the proof of Theorem~\ref{thm:LCLT} implies
\[  \sup_{k \in \mathbb{N}}{\Big|} v_n(k) - \alpha_n(k) {\Big|} = o(n^{-1/2}), \]
where
\[ \alpha_n(k) = \frac{C_2}{\sqrt{2\pi n \sigma_2}}e^{\frac{-(k-n/2)^2}{2n\sigma_2}} \;\text{ for }\; C_2 = \frac{1}{4\log 2} \quad\text{and}\quad \sigma_2 = \frac{1-\log2}{4}. \]
Furthermore, if $w_n(k) = 2^{-n}\binom{n}{k}$ then it is classical that
\[  \sup_{k \in \mathbb{N}}{\Big|} w_n(k) - \beta_n(k) {\Big|} = o(n^{-1/2}) \]
where
\[ \beta_n(k) = \frac{1}{\sqrt{\pi n/2}}e^{\frac{-(k-n/2)^2}{n/2}}. \]
Our goal is to find a limit theorem for $\frac{(2\log2)^n}{n!} \; D_{n-k,k} = \frac{v_n(k)}{w_n(k)}$. Note that for all $k$,
\begin{align*}
\left|\frac{v_n(k)}{w_n(k)} - \frac{\alpha_n(k)}{\beta_n(k)} \right|
&= \left|\frac{v_n(k)}{w_n(k)} - \frac{\alpha_n(k)}{w_n(k)} + \frac{\alpha_n(k)}{w_n(k)}  - \frac{\alpha_n(k)}{\beta_n(k)} \right| \\[+2mm]
&\leq \left|\frac{v_n(k)}{w_n(k)} - \frac{\alpha_n(k)}{w_n(k)}\right|
+ \left|\frac{\alpha_n(k)}{w_n(k)}  - \frac{\alpha_n(k)}{\beta_n(k)}\right| \\[+2mm]
&= \frac{\left|v_n(k) - \alpha_n(k)\right|}{w_n(k)}
+ \frac{\alpha_n(k)}{w_n(k)\beta_n(k)} {\Big|}w_n(k)-\beta_n(k){\Big|}.
\end{align*}
Since $\left|v_n(k) - \alpha_n(k)\right|$ and $|w_n(k)-\beta_n(k)|$ are $o(n^{-1/2})$, and both 
\[ \frac{1}{w_n(k)} = \frac{2^n}{\binom{n}{k}} \]
and
\[ \frac{\alpha_n(k)}{w_n(k)\beta_n(k)} = \frac{2^{n-1}}{\binom{n}{k}4\log 2 \sqrt{\sigma_2}}\exp\left[-(k-n/2)^2 \cdot \frac{1-4\sigma_2}{2n\sigma_2}\right] \]
are $O(\sqrt{n})$ when $k = n/2 + O(\sqrt{n})$. Simplifying the ratio $\alpha_n(k)/\beta_n(k)$ gives the stated result.
\end{proof}

\bibliographystyle{abbrv}
\bibliography{ref}

\begin{thebibliography}{10}

\bibitem{Ahlfors}
L.~V. Ahlfors.
\newblock {\em Complex analysis}.
\newblock McGraw-Hill Book Co., New York, third edition, 1978.
\newblock An introduction to the theory of analytic functions of one complex
  variable, International Series in Pure and Applied Mathematics.

\bibitem{ArakawaKaneko1999}
T.~Arakawa and M.~Kaneko.
\newblock Multiple zeta values, poly-{B}ernoulli numbers, and related zeta
  functions.
\newblock {\em Nagoya Math. J.}, 153:189--209, 1999.

\bibitem{Kaneko2}
T.~Arakawa and M.~Kaneko.
\newblock On poly-{B}ernoulli numbers.
\newblock {\em Comment. Math. Univ. St. Paul.}, 48(2):159--167, 1999.

\bibitem{Baryshnikov}
Y.~Baryshnikov and R.~Pemantle.
\newblock Asymptotics of multivariate sequences, part {III}: {Q}uadratic
  points.
\newblock {\em Adv. Math.}, 228(6):3127--3206, 2011.

\bibitem{benyi}
B.~B\'{e}nyi and P.~Hajnal.
\newblock Combinatorics of poly-{B}ernoulli numbers.
\newblock {\em Studia Sci. Math. Hungar.}, 52(4):537--558, 2015.

\bibitem{benyihajnal}
B.~B\'enyi and P.~Hajnal.
\newblock Combinatorial properties of poly-{B}ernoulli relatives.
\newblock {\em Integers}, 17:Paper No. A31, 26, 2017.

\bibitem{Brewbaker}
C.~Brewbaker.
\newblock A combinatorial interpretation of the poly-bernoulli numbers and two
  fermat analogues.
\newblock {\em Integers}, 8(1):Article A02, 9 p., electronic only--Article A02,
  9 p., electronic only, 2008.

\bibitem{CameronPreprint}
P.~J. Cameron, C.~A. Glass, and R.~U. Schumacher.
\newblock Acyclic orientations and poly-bernoulli numbers.
\newblock {\em preprint, arXiv:1412.3685}, 2014.

\bibitem{ALN}
R.~F. de~Andrade, E.~Lundberg, and B.~Nagle.
\newblock Asymptotics of the extremal excedance set statistic.
\newblock {\em European J. Combin.}, 46:75--88, 2015.

\bibitem{deBruijn}
N.~G. de~Bruijn.
\newblock {\em Asymptotic methods in analysis}.
\newblock Dover Publications, Inc., New York, third edition, 1981.

\bibitem{EhrSt}
R.~Ehrenborg and E.~Steingr\'{\i}msson.
\newblock The excedance set of a permutation.
\newblock {\em Adv. in Appl. Math.}, 24(3):284--299, 2000.

\bibitem{FRS}
A.~Fink, J.~Rajchgot, and S.~Sullivant.
\newblock Matrix {S}chubert varieties and {G}aussian conditional independence
  models.
\newblock {\em J. Algebraic Combin.}, 44(4):1009--1046, 2016.

\bibitem{Sullivant}
S.~Ho\c{s}ten and S.~Sullivant.
\newblock The algebraic complexity of maximum likelihood estimation for
  bivariate missing data.
\newblock In {\em Algebraic and geometric methods in statistics}, pages
  123--133. Cambridge Univ. Press, Cambridge, 2010.

\bibitem{Kaneko}
M.~Kaneko.
\newblock Poly-{B}ernoulli numbers.
\newblock {\em J. Th\'{e}or. Nombres Bordeaux}, 9(1):221--228, 1997.

\bibitem{Cai}
W.~Letsou and L.~Cai.
\newblock Noncommutative biology: Sequential regulation of complex networks.
\newblock {\em PLOS Computational Biology}, 12:1--36, 08 2016.

\bibitem{Lovasz}
L.~Lov\'asz and K.~Vesztergombi.
\newblock Restricted permutations and {S}tirling numbers.
\newblock In {\em Combinatorics ({P}roc. {F}ifth {H}ungarian {C}olloq.,
  {K}eszthely, 1976), {V}ol. {II}}, volume~18 of {\em Colloq. Math. Soc.
  J\'anos Bolyai}, pages 731--738. North-Holland, Amsterdam-New York, 1978.

\bibitem{MelczerMishna2016}
S.~Melczer and M.~Mishna.
\newblock Asymptotic lattice path enumeration using diagonals.
\newblock {\em Algorithmica}, 75(4):782--811, 2016.

\bibitem{MelczerWilson2019}
S.~Melczer and M.~C. Wilson.
\newblock Higher {D}imensional {L}attice {W}alks: {C}onnecting {C}ombinatorial
  and {A}nalytic {B}ehavior.
\newblock {\em SIAM J. Discrete Math.}, 33(4):2140--2174, 2019.

\bibitem{PemWilTwenty}
R.~Pemantle and M.~C. Wilson.
\newblock Twenty combinatorial examples of asymptotics derived from
  multivariate generating functions.
\newblock {\em SIAM Rev.}, 50(2):199--272, 2008.

\bibitem{PemantleWilson}
R.~Pemantle and M.~C. Wilson.
\newblock {\em Analytic combinatorics in several variables}, volume 140 of {\em
  Cambridge Studies in Advanced Mathematics}.
\newblock Cambridge University Press, Cambridge, 2013.

\bibitem{Ryser}
H.~J. Ryser.
\newblock Combinatorial properties of matrices of zeros and ones.
\newblock {\em Canadian J. Math.}, 9:371--377, 1957.

\bibitem{Stanley}
R.~P. Stanley.
\newblock {\em Enumerative combinatorics. {V}olume 1}, volume~49 of {\em
  Cambridge Studies in Advanced Mathematics}.
\newblock Cambridge University Press, Cambridge, second edition, 2012.

\bibitem{VesztergombiPerm}
K.~Vesztergombi.
\newblock Permutations with restriction of middle strength.
\newblock {\em Studia Sci. Math. Hungarica}, 9:181--185, 1974.

\end{thebibliography}

\end{document}